\documentclass[12pt, a4paper]{article}

\usepackage[utf8]{inputenc}
\usepackage[T1]{fontenc}
\usepackage[english]{babel}
\usepackage{amsmath, amssymb, amsthm}
\usepackage{mathtools}
\usepackage{geometry}
\usepackage{enumitem}
\usepackage{hyperref}
\usepackage{booktabs}
\usepackage{tabularx}
\usepackage{calc}
\usepackage{listings}
\usepackage[ruled,vlined,linesnumbered]{algorithm2e} 

\theoremstyle{plain}
\newtheorem{theorem}{Theorem}[section]
\newtheorem{proposition}[theorem]{Proposition}
\newtheorem{lemma}[theorem]{Lemma}
\newtheorem{corollary}[theorem]{Corollary}
\newtheorem{conjecture}[theorem]{Conjecture}

\theoremstyle{definition}
\newtheorem{definition}[theorem]{Definition}
\newtheorem{problem}[theorem]{Problem}

\theoremstyle{remark}

\newcommand{\Shat}{\widehat{S}}
\newcommand{\Irr}{\operatorname{Irr}}
\newcommand{\Aut}{\operatorname{Aut}}
\newcommand{\Inn}{\operatorname{Inn}}
\newcommand{\Out}{\operatorname{Out}}
\newcommand{\Syl}{\mathrm{Syl}}

\newcommand{\N}{\operatorname{N}}
\newcommand{\Z}{\operatorname{Z}}
\newcommand{\normal}{\trianglelefteq}
\newcommand{\Char}{\operatorname{Char}}
\newcommand\blfootnote[1]{%
    \begingroup
    \renewcommand\thefootnote{}\footnote{#1}%
    \addtocounter{footnote}{-1}%
    \endgroup
}

\title{\textbf{On the inductive Giannelli--McKay condition for sporadic groups}}
\author{Izan Gómez Taberner}
\date{\today}

\begin{document}
\maketitle

\begin{abstract}
We present a computational framework to verify a refinement of the McKay conjecture, proposed by E. Giannelli, for sporadic simple groups and their universal covers. This refinement asserts the existence of a bijection between the $p'$-characters of a finite group and those of the normalizer of its Sylow $p$-subgroup that does not increase character degrees. Our verification enforces stricter conditions, ensuring equivariance under the relevant outer automorphism group, preservation of the $p$-part of the character conductor, and compatibility with central characters. We establish theoretical reductions that limit the relevant outer automorphism action to an order of at most 2, allowing us to leverage Clifford theory and precomputed character tables to circumvent expensive explicit group constructions. Finally, we provide algorithms for classifying characters and constructing bijections that strictly satisfy these conditions.
\end{abstract}

\section{Introduction}

Let $G$ be a finite group, $p$ a prime and $P$ a Sylow $p$-subgroup of $G$. We denote by $\mathrm{Irr}_{p'}(G)$ the set of irreducible complex characters of $G$ whose degree $\chi(1)$ is not divisible by $p$. The McKay conjecture states that
$$|\mathrm{Irr}_{p'}(G)|=|\Irr_{p'}(\N_{G}(P))|.$$
In \cite{IMN}, the McKay conjecture was reduced to proving the existence of so-called \emph{McKay-good bijections} for quasi-simple groups. After work of many authors over the last 20 years, the final proof of the conjecture appeared in \cite{CS} using the Classification of Finite Simple Groups and \cite{IMN}.

The McKay conjecture was the first of the local-global counting conjectures, and has inspired a range of deeper statements. For instance, J. L. Alperin \cite{A} proposed a refinement involving height zero characters in Brauer $p$-blocks (this is now known as the Alperin--McKay conjecture), and E. C. Dade stated a series of conjectures involving chains of $p$-subgroups which generalized both of these conjectures \cite{D}. Further, in \cite{N}, G. Navarro proposed a refinement of the McKay and Alperin--McKay conjectures that involved the action of certain Galois automorphisms, and A. Turull conjectured a version with Schur indices in \cite{T}. However, this paper concerns a recent refinement of the McKay conjecture proposed by E. Giannelli \cite{G}\blfootnote{This work was supported by a Collaboration Scholarship from the Ministry of Education (2025-2026), carried out under the supervision of Professor J. Miquel Martínez at the Universitat de València during the 2025-2026 academic year.}.

\begin{conjecture}[E. Giannelli]
Let $G$ be a finite group, $p$ a prime and $P$ a Sylow $p$-subgroup of $G$. Then there is a bijection \[\varepsilon:\mathrm{Irr}_{p'}(G)\rightarrow\Irr_{p'}(\N_{G}(P))\]
such that $\varepsilon(\chi)(1)\leq\chi(1)$ for all $\chi\in\mathrm{Irr}_{p'}(G)$.
\end{conjecture}

In \cite{G}, Giannelli also proves his conjecture for the alternating and symmetric groups, and the case of $p$-solvable groups was known beforehand \cite{R}. Giannelli's refinement would be consistent with the existence of a McKay bijection $f:\mathrm{Irr}_{p'}(G)\rightarrow\Irr_{p'}(\N_{G}(P))$ where $f(\chi)$ is contained in the restriction $\chi_{\N_G(P)}$ for all $\chi\in\mathrm{Irr}_{p'}(G)$. This is however untrue in general, as shown by the alternating group $\mathsf{A}_5$ with $p=2$.

Following the approach from \cite{IMN}, Giannelli's conjecture was reduced to a problem on finite simple groups in \cite{HMN}. More precisely, it is shown in \cite[Section 3]{HMN} that Giannelli's conjecture would follow from the existence of McKay--good bijections $\Psi:\mathrm{Irr}_{p'}(X)\rightarrow\mathrm{Irr}_{p'}(\N_X(Q))$ for quasisimple groups $X$ with $Q\in\Syl_p(X)$, with the extra condition that $\Psi(\chi)(1)\leq \chi(1)$. This is ultimately proved for $p=2$ in \cite{HMN}, and, in some other cases, for some families of finite quasi-simple groups, including the sporadic groups. We mention that, in many cases (such as the case of sporadic groups), the verification of this conjecture relies on proving that any nontrivial character in $\Irr_{p'}(X)$ has a larger degree than the largest character degree in $\Irr_{p'}(\N_X(Q))$, and then applying \cite{CS}. This is not always true, as pointed out in \cite[Section 4]{HMN}.

The objective of this note is to show a computational approach to check a stronger version of this problem for sporadic groups with \cite{GAP}. The precise condition that we check is stated in Section \ref{sec:inductive-condition}. In fact, the McKay-good bijections that are checked with the algorithm proposed in this note also preserve the $p$-part of the character conductor, a behavior that would certainly agree with the Galois refinement of the McKay conjecture from \cite{N}. We mention that our approach requires the construction in \cite{GAP} of the automorphism groups of universal covers of sporadic groups. More on the limitations of this approach is discussed in Section \ref{sec:PC_specifications}. On a more positive note, our algorithm only relies on the fact that, if $X$ is the universal cover of a sporadic group, then $\Out(X)$ is trivial or cyclic of prime order. Therefore, it would only require minor modifications to computationally check quasisimple groups satisfying this condition on the outer automorphism group such as the central extensions of the simple alternating groups with $n\geq 7$ and some groups of Lie type (for instance, some groups of the families $E_8, F_4$ and $G_2$), as long as they can be constructed in \cite{GAP}.
 
The paper is structured as follows. In Section \ref{sec:inductive-condition}, we formalize the Refined McKay Problem for sporadic covers. Section \ref{sec:reducciones_teoricas} establishes the basic notation and provides theoretical reductions, demonstrating that the relevant action factors through an outer quotient of order at most $2$. In Section \ref{sec:clifford}, we rely on Clifford theory to introduce witness patterns and prove the Library Shortcut Theorem, which allows us to classify character types using precomputed character tables rather than explicit group constructions. Section \ref{sec:computational_framework} details the computational framework, including the resource allocation strategy and the Block-and-Match algorithm used to construct the bijections. Section \ref{sec:PC_specifications} outlines the computer specifications used for these verifications.

Finally, an example of an explicit classification of irreducible characters and the corresponding bijection for the sporadic group $\mathrm{Fi}_{24}'$ and prime $p=5$ produced by our code is compiled in Appendix \ref{app:bijection_example}. The GAP code and the output for most of the computed cases are publicly available at \url{https://github.com/IzanGT/GAP_github}. We note that some cases could not be computed due to hardware limitations (specifically, RAM exhaustion).

\section{The inductive condition}\label{sec:inductive-condition}

Throughout, our notation for finite groups follows \cite{Is} and our notation for characters follows \cite{Nbook}.

We focus only on the case when $X$ is a quasisimple group with $\Out(X)$ trivial or of prime order. Following the proof of \cite[Proposition 3.17]{Spa}, it is straightforward to show that the inductive condition for the Giannelli conjecture \cite[Conjecture 3.4]{HMN} is equivalent to conditions (C1)-(C3) below. 

\begin{problem}[Refined McKay Problem for Sporadic Covers]\label{inductive-condition}
\label{prob:mckay}
Let $S$ be a sporadic simple group with universal cover $\Shat$, let $p$ be a prime, and let $P \in \Syl_p(\Shat)$. Does 
there exist an explicit bijection 
\[
{}^{*} \colon \Irr_{p'}(\Shat) \longrightarrow \Irr_{p'}(\N_{\Shat}(P))
\]
satisfying the following four conditions?
\begin{enumerate}[label=\textnormal{(C\arabic*)}]
  \item \textbf{Equivariance.} $(\chi^*)^\alpha = (\chi^\alpha)^*$ for every $\alpha \in \Aut(\Shat)_P$.
  \item \textbf{Non-increasing degree.} $\chi^*(1) \le \chi(1)$.
  \item \textbf{Central character (Schur).} $\chi_{\Z(\Shat)} = \chi(1) \lambda \iff (\chi^*)_{\Z(\Shat)} = \chi^*(1) \lambda$ for the same $\lambda \in 
  \Irr(\Z(\Shat))$.
  \item \textbf{$p$-part of the conductor.} $C(\chi)_p = C(\chi^*)_p$.
\end{enumerate}
\end{problem}

\section{Notation and Theoretical Reductions}
\label{sec:reducciones_teoricas}

Let $G$ be a finite group. We use standard exponential notation for the action of automorphisms: for $\chi \in \Char(G)$ and $\alpha \in \Aut(G)$, $\chi^\alpha(g) 
:= \chi(\alpha^{-1}(g))$. The inertia group of $\theta \in \Irr(N)$ for $N \normal G$ is denoted $G_\theta$.

\begin{definition}[Universal Cover and Schur Multiplier]
\label{def:universal-cover}
Let $S$ be a finite non-abelian simple group. Its \emph{universal cover} $\Shat$ is the unique (up to isomorphism) finite group satisfying the following conditions:
\begin{enumerate}[label=(\roman*)]
  \item $\Shat$ is perfect;
  \item $\Shat / \Z(\Shat) \cong S$;
  \item $\Shat$ is the largest group satisfying (i) and (ii).
\end{enumerate}
The order of the center, $m(S) := |\Z(\Shat)|$, is called the \emph{Schur multiplier} of $S$. For sporadic simple groups, it is known that $m(S) \in 
\{1, 2, 3, 4, 6, 12\}$.
\end{definition}

\begin{definition}[Conductor of a Character]
\label{def:conductor}
For a character $\chi \in \Irr(G)$, let $\mathbb{Q}(\chi)$ be its field of values, which is the smallest subfield of $\mathbb{C}$ containing $\chi(g)$ for all 
$g \in G$. By a theorem of Brauer, $\mathbb{Q}(\chi)$ is contained in some cyclotomic field. The \emph{conductor} $C(\chi)$ is defined as the smallest positive 
integer $n$ such that $\mathbb{Q}(\chi) \subseteq \mathbb{Q}(\zeta_n)$, where $\zeta_n$ is a primitive $n$-th root of unity. Furthermore, the \emph{$p$-part of the 
conductor}, denoted $C(\chi)_p$, is the largest power of the prime $p$ dividing $C(\chi)$.
\end{definition}

\begin{definition}[Relevant outer quotient]
Let $P \in \Syl_p(\Shat)$. The relevant outer quotient for our problem is defined as $\Gamma_P := \Aut(\Shat)_P / \Inn(\Shat)_P$, where $\Aut(\Shat)_P = \{ \alpha 
\in \Aut(\Shat) : \alpha(P) = P \}$ and $\Inn(\Shat)_P = \Inn(\Shat) \cap \Aut(\Shat)_P$.
\end{definition}

Condition (C1) presents high computational complexity when evaluating the full group $\Aut(\Shat)_P$. However, the action reduces to the outer automorphism group 
since inner automorphisms act trivially on characters.

\begin{lemma}[The action factors through $\Aut/\Inn$]
\label{lem:accion-outer}
Let $\alpha, \beta \in \Aut(G)$ lie in the same coset modulo $\Inn(G)$. Then $\chi^\alpha = \chi^\beta$ for every $\chi \in \Irr(G)$. In particular, the action of 
$\Aut(G)$ on $\Irr(G)$ factors through $\Aut(G)/\Inn(G)$.
\end{lemma}
\begin{proof}
Write $\beta = \iota_g \circ \alpha$ for some $g \in G$. Since irreducible characters are class functions, $\chi^{\iota_g} = \chi$. Thus, evaluating the action 
yields $\chi^\beta = \chi^{\iota_g \circ \alpha} = (\chi^\alpha)^{\iota_g} = \chi^\alpha$.
\end{proof}

For sporadic groups, the structure of $\Gamma_P$ is constrained. 

\begin{theorem}
\label{thm:gamma-P}
For every sporadic simple group $S$, every prime $p$, and every $P \in \Syl_p(\Shat)$, the group $\Gamma_P$ has order at most $2$.
\end{theorem}
\begin{proof}
By the Frattini argument, since all Sylow $p$-subgroups of $\Shat$ are conjugate, we have $\Aut(\Shat) = \Inn(\Shat)\Aut(\Shat)_P$. By the Second Isomorphism Theorem, 
the quotient $\Gamma_P = \Aut(\Shat)_P / \Inn(\Shat)_P$ is exactly isomorphic to $\Aut(\Shat)/\Inn(\Shat) = \Out(\Shat)$. For sporadic groups, it is well known from the ATLAS that $|\Out(S)| \in \{1, 2\}$. Since the universal cover satisfies $\Out(\Shat) \cong \Out(S)$, it follows that $\Gamma_P$ has order at most $2$.
\end{proof}

This structural constraint simplifies the verification of (C1), reducing it to a single equation.

\begin{corollary}[Reduction to one $\alpha$]
\label{cor:un-alpha}
If $|\Gamma_P| = 1$, condition (C1) is automatic. If $|\Gamma_P| = 2$, choose any $\alpha \in \Aut(\Shat)_P$ whose image is the unique non-trivial element in 
$\Gamma_P$. Then (C1) is equivalent to the single equation $(\chi^*)^\alpha = (\chi^\alpha)^*$ for all $\chi \in \Irr_{p'}(\Shat)$.
\end{corollary}
\begin{proof}
Every $\beta \in \Aut(\Shat)_P$ is either inner, making (C1) automatic by Lemma \ref{lem:accion-outer}, or can be written as $\beta = \alpha \circ \sigma$ for 
some $\sigma \in \Inn(\Shat)_P$. By Lemma \ref{lem:accion-outer}, $\sigma$ acts trivially, so $(\chi^*)^\beta = (\chi^*)^\alpha$ and $(\chi^\beta)^* = 
(\chi^\alpha)^*$. The equation $(\chi^*)^\alpha = (\chi^\alpha)^*$ thus implies (C1).
\end{proof}

From now on, when $|\Gamma_P| = 2$, let $\alpha$ denote this chosen representative. Because $\alpha^2 \in \Inn(\Shat)_P$, it acts trivially. This partitions 
$\Irr(\Shat)$ into two types:
\begin{itemize}
  \item \textbf{Type 1:} $\theta^\alpha = \theta$.
  \item \textbf{Type 2:} $\theta^\alpha \neq \theta$, and the $\langle \alpha \rangle$-orbit of $\theta$ is $\{\theta, \theta^\alpha\}$.
\end{itemize}
Any bijection ${}^*$ satisfies (C1) if and only if it preserves these types and maps each $\langle \alpha \rangle$-orbit on $\Shat$ to an $\langle \alpha 
\rangle$-orbit on $\N_{\Shat}(P)$.

\begin{proposition}[Independence of the Sylow $p$-subgroup]
  \label{prop:independence_sylow}
  The type associated with these characters is independent of the chosen Sylow $p$-subgroup.
\end{proposition}
\begin{proof}
  If $P' = P^g$ is another Sylow $p$-subgroup of $\Shat$ with $g \in \Shat$, then $\alpha' := \iota_g \circ \alpha \circ \iota_g^{-1}$ stabilizes $P'$ and represents a non-trivial class of $\Gamma_{P'}$. By Lemma \ref{lem:accion-outer}, the corresponding partition of $\Irr_{p'}(\Shat)$ is the same.
\end{proof}

Thus the type assignment depends only on $\Shat$ and $p$, not on the choice of the Sylow $p$-subgroup, which may be chosen arbitrarily.

To assign a type to these characters computationally without constructing the explicit semidirect product, we rely on computing the permutation of conjugacy classes induced by the outer automorphism $\alpha$. The procedure is detailed in Algorithm \ref{algo:standard_path}.

\begin{algorithm}[H]
\caption{Standard Path for Character Classification}\label{algo:standard_path}
\SetKwInOut{Input}{Input}
\SetKwInOut{Output}{Output}
\Input{A finite group $G$ (representing either $\Shat$ or $\N_{\Shat}(P)$) and a lift $\alpha$ of the non-trivial class of $\Gamma_P$ (or \texttt{fail} if no outer action exists).}
\Output{Partition of $\Irr_{p'}(G)$ into Type 1 singletons and Type 2 pairs.}

\If{$\alpha = \text{\normalfont\texttt{fail}}$}{
    Mark all $\chi \in \Irr_{p'}(G)$ as \textbf{Type 1} and \Return\;
}
Compute the permutation $\sigma$ of the conjugacy classes of $G$ such that $\alpha(C_j) = C_{\sigma(j)}$\;
Initialize a dictionary mapping the string representation of character values to their index in $\Irr(G)$ to avoid quadratic search complexity\;
\For{each $\chi_i \in \Irr(G)$}{
    Compute the values of the conjugate character $\chi_i^\alpha$ by evaluating $\chi_i^\alpha(C_j) = \chi_i(C_{\sigma^{-1}(j)})$\;
    Query the dictionary to identify the index $k$ matching the values of $\chi_i^\alpha$ in $\Irr(G)$\;
    Record the permutation of characters: $\pi(i) = k$\;
}
\For{each $\chi_i \in \Irr_{p'}(G)$}{
    \eIf{$\pi(i) = i$}{
        Mark $\chi_i$ as \textbf{Type 1}\;
    }{
        \If{$\pi(i) = j$ with $j \neq i$}{
            Mark $\chi_i$ and $\chi_j$ as \textbf{Type 2} and pair them as an orbit $\{i, j\}$\;
        }
    }
}
\end{algorithm}

Hereafter, both in the text and in the implementation, this method of determining the type of each character will be referred to as the \textbf{Standard Path}. The constraint is obvious, since this path requires the computer to calculate expensive groups, conjugacy classes, and their permutations.

\section{Clifford Witnesses and the Library Shortcut}
\label{sec:clifford}

To determine the character types without directly computing the automorphism action, we can embed $\Shat$ in a larger group and analyze the restriction of characters 
from this larger group back to $\Shat$. This approach relies on Clifford theory and it serves both to verify the correctness of the result through an entirely distinct approach and to provide the underlying principle behind the shortcut that, in many cases, allows us to bypass the full group construction by leveraging precomputed character tables from the library \texttt{CTblLib}.

We define the cyclic extension $A := \Shat \rtimes \langle \alpha \rangle$, where $\alpha \in \Aut(\Shat)_P$ is the chosen representative for 
the non-trivial class of $\Gamma_P$. The following proposition shows how restriction from $A$ to $\Shat$ behaves, establishing predictable restriction patterns.

\begin{proposition}[Classification via the extension]
\label{prop:clasificacion-A-indice-n}
Let $A = \Shat \rtimes \langle \alpha \rangle$, where $\alpha$ represents the non-trivial class of $\Gamma_P$, and let $\theta \in \Irr(\Shat)$.
Then $\theta$ is of \emph{Type 1} if and only if there exists $\psi \in \Irr(A)$ that restricts to $\theta$ (\emph{Pattern a}), 
and it is of \emph{Type 2} if and only if there exists $\psi \in \Irr(A)$ that restricts to $\theta + \theta^\alpha$ (\emph{Pattern b}).
\end{proposition}
\begin{proof}
Because $\overline{\alpha}$ has order 2 in $\Out(\Shat)$ and $\alpha \notin \Inn(\Shat)$, the order $n$ of $\alpha$ must be
 even, with $\alpha^2 \in \Inn(\Shat)$. Let $B = \Shat \langle \alpha^2 \rangle$. Since $A/\Shat \cong C_n$ is cyclic and $n$ is even, $B/\Shat$ is the
 \emph{unique} subgroup of index 2 in $A/\Shat$, and $B$ is a maximal subgroup of $A$.

By Lemma \ref{lem:accion-outer}, $\alpha^2 \in \Inn(\Shat)$ acts trivially on $\Irr(\Shat)$. Therefore, every $\langle \alpha \rangle$-orbit on $\Irr(\Shat)$ 
has size at most 2, and $B \le A_\theta$ for every $\theta \in \Irr(\Shat)$.

If $\theta$ is of Type 1 ($\theta^\alpha = \theta$), it is fully $A$-invariant. Since the quotient $A/\Shat$ is cyclic, standard Clifford theory dictates that 
$\theta$ extends to $A$. Thus, there exists $\psi \in \Irr(A)$ with $\psi_{\Shat} = \theta$, establishing Pattern (a).

If $\theta$ is of Type 2 ($\theta^\alpha \neq \theta$), then $\alpha \notin A_\theta$. We have the chain of subgroups $\Shat \le B \le A_\theta \subsetneq A$. 
Because $B$ is maximal in $A$, the inertia subgroup is exactly $A_\theta = B$. Because $B/\Shat$ is cyclic, $\theta$ extends to some $\xi \in \Irr(B)$, 
satisfying $\xi(1) = \theta(1)$. By the Clifford correspondence, $\psi := \xi^A$ is irreducible in $A$. Furthermore, standard induction degrees yield $\psi(1) 
= [A : B]\xi(1) = 2\theta(1)$. 

By Clifford's restriction theorem applied to $\Shat \normal A$, $\psi_{\Shat} = e(\theta + \theta^\alpha)$ for some multiplicity $e \ge 1$. Checking degrees, 
we have $\psi(1) = e(2\theta(1))$. Substituting $\psi(1) = 2\theta(1)$ forces $e = 1$. Thus, $\psi_{\Shat} = \theta + \theta^\alpha$, establishing Pattern (b).
\end{proof}

As established in Proposition \ref{prop:clasificacion-A-indice-n}, the explicit semidirect product $A = \Shat \rtimes \langle \alpha \rangle$ perfectly controls the 
inertia groups and guarantees that the Clifford multiplicities are exactly 1. 

\begin{definition}[Clifford Witnesses]
\label{def:witnesses}
Let $\alpha \in \Aut(\Shat)_P$ represent the non-trivial class in $\Gamma_P$, and let $A = \Shat \rtimes \langle \alpha \rangle$. For any irreducible character 
$\theta \in \Irr_{p'}(\Shat)$:
\begin{itemize}
    \item A character $\psi \in \Irr(A)$ satisfying $\psi_{\Shat} = \theta$ is called a \textbf{Type 1 witness} for $\theta$.
    \item A character $\psi \in \Irr(A)$ satisfying $\psi_{\Shat} = \theta + \theta^\alpha$ is called a \textbf{shared Type 2 witness} for the orbit 
    $\{\theta, \theta^\alpha\}$.
\end{itemize}
\end{definition}

This theoretical foundation yields an algorithm to verify character types explicitly. When verification is requested, Algorithm \ref{algo:verification} constructs the semidirect product to audit the character restrictions.

\begin{algorithm}[H]
\caption{Clifford Verification via Explicit Semidirect Product}\label{algo:verification}
\SetKwInOut{Input}{Input}
\SetKwInOut{Output}{Output}
\Input{A finite group $\Shat$, a Sylow $p$-subgroup $P$, and a previously computed partition of $\Irr_{p'}(\Shat)$.}
\Output{Verification of the Type 1 / Type 2 partition.}

Select a lift $\alpha \in \Aut(\Shat)_P$ representing the non-trivial class of $\Gamma_P$\;
Construct the semidirect product $A := \Shat \rtimes \langle \alpha \rangle$\;
Compute the restriction map $\text{fus}: \Shat \hookrightarrow A$\;
\For{each $\psi \in \Irr(A)$}{
    Restrict $\psi$ to $\Shat$ via the map $\text{fus}$\;
    \If{$[\psi_{\Shat}, \theta] = 1$ for a unique $\theta \in \Irr_{p'}(\Shat)$}{
        Confirm $\theta$ possesses a \textbf{Type 1} witness\;
    }
    \ElseIf{$[\psi_{\Shat}, \theta_1] = 1$ and $[\psi_{\Shat}, \theta_2] = 1$ for distinct $\theta_1, \theta_2 \in \Irr_{p'}(\Shat)$}{
        Confirm $\{\theta_1, \theta_2\}$ possesses a \textbf{Type 2} witness\;
    }
}
Assert that every p'-character matches its previously computed type via a corresponding witness\;
\end{algorithm}

\subsection{The Library Shortcut}

Proposition \ref{prop:clasificacion-A-indice-n} demonstrates that the internal order $n$ of the lift $\alpha$ does not affect the restriction patterns. The restriction patterns are fully determined by the outer action.

Consequently, the construction of $A = \Shat \rtimes \langle \alpha \rangle$ can be substituted with any index-2 extension realizing the same outer 
action, obtaining the exact same partition of $\Irr_{p'}(\Shat)$.

\begin{theorem}[Library Shortcut for $\Shat$]
\label{thm:atajo-S}
Let $H$ be a finite group with $\Shat \normal H$ and $[H : \Shat] = 2$, such that conjugation by elements of $H \setminus \Shat$ induces the 
non-trivial outer class $\overline{\alpha} \in \Gamma_P$. The restriction map $\Irr(H) \to \Char(\Shat)$ correctly produces the
Type 1 / Type 2 partition with witness patterns (a) and (b), identical to those produced by $A = \Shat \rtimes \langle \alpha \rangle$.
\end{theorem}
\begin{proof}
By Lemma \ref{lem:accion-outer}, the action of $H/\Shat$ on $\Irr(\Shat)$ is exactly the action of $\overline{\alpha}$. Let $\theta 
\in \Irr(\Shat)$. Because $[H : \Shat] = 2$, the inertia group $H_\theta$ must be strictly either $H$ or $\Shat$.

If $\theta$ is of Type 1 ($\theta^\alpha = \theta$), then $H_\theta = H$. Since the quotient $H/\Shat \cong C_2$ is cyclic, standard 
Clifford theory dictates that $\theta$ extends to $H$. Thus, there exists $\psi \in \Irr(H)$ such that $\psi_{\Shat} = \theta$, which exactly matches Pattern (a).

If $\theta$ is of Type 2 ($\theta^\alpha \neq \theta$), then $H_\theta = \Shat$. By the Clifford correspondence, $\theta$ induces to an irreducible 
character $\psi = \theta^H \in \Irr(H)$. Standard induction degrees give $\psi(1) = [H : \Shat]\theta(1) = 2\theta(1)$. By Clifford's restriction theorem 
applied to $\Shat \normal H$, we know $\psi_{\Shat} = e(\theta + \theta^\alpha)$ for some multiplicity $e \ge 1$. Comparing degrees yields $2\theta(1) = 
e(2\theta(1))$, which strictly forces $e = 1$. Thus, $\psi_{\Shat} = \theta + \theta^\alpha$, matching Pattern (b).
\end{proof}

Theorem \ref{thm:atajo-S} optimizes the classification process. When the character tables for $\Shat$ and a suitable index-2 extension $H$ are 
available in a computational library, the construction of the explicit semidirect product $A$ can be avoided.

To extend this to the normalizer side, we note that if $H$ is an index-2 extension of $\Shat$, it yields an index-2 extension for the Sylow normalizer. 
By a standard Frattini argument, for any $P \in \Syl_p(\Shat)$, $H = \Shat \cdot \N_H(P)$. Consequently, $\N_{\Shat}(P) \normal \N_H(P)$ and 
\[
\N_H(P)/\N_{\Shat}(P) \cong H/\Shat \cong C_2.
\]
Because $\N_H(P)$ is an index-2 extension of $\N_{\Shat}(P)$ that induces the same outer action, the restriction map $\Irr(\N_H(P)) \to \Irr(\N_{\Shat}(P))$ 
also classifies the character types via identical witness patterns. 

Thus, the character types can be entirely deduced from the restriction of characters in a suitable index-2 extension, bypassing the internal geometry of the 
specific semidirect product $A$.

Constructing the universal cover $\Shat$, computing its full automorphism group to extract the outer generator $\alpha$, and building the explicit semidirect 
product $A$ can be computationally expensive. To mitigate this, our computational framework utilizes the Library Shortcut (Theorem \ref{thm:atajo-S}). 
By accessing precomputed character tables from the \texttt{CTblLib} package in GAP, expensive group-theoretic constructions are replaced with linear algebra operations over character tables. 

When the character tables for $\Shat$ and a suitable index-2 extension $H$ are available, along with the corresponding class fusion map $\text{fus}: 
\Shat \to H$, the classification is executed directly via Algorithm \ref{algo:library}. 

\begin{algorithm}[H]
\caption{Character Classification via Library Tables}\label{algo:library}
\SetKwInOut{Input}{Input}
\SetKwInOut{Output}{Output}
\Input{Precomputed character tables $T_{\Shat}$ and $T_H$ (where $[H:\Shat]=2$), and the fusion map $\text{fus}: \Shat \hookrightarrow H$.}
\Output{Partition of $\Irr_{p'}(\Shat)$ into Type 1 singletons and Type 2 pairs.}

\For{each $\psi \in \Irr(T_H)$}{
    Restrict $\psi$ to $\Shat$ via the map $\text{fus}$\;
    \If{$[\psi_{\Shat}, \theta] = 1$ for a unique $\theta \in \Irr_{p'}(T_{\Shat})$}{
        Mark $\theta$ as \textbf{Type 1}\;
    }
    \ElseIf{$[\psi_{\Shat}, \theta_1] = 1$ and $[\psi_{\Shat}, \theta_2] = 1$ for distinct $\theta_1, \theta_2 \in \Irr_{p'}(T_{\Shat})$}{
        Mark $\theta_1$ and $\theta_2$ as \textbf{Type 2} and pair them as an orbit\;
    }
}
\end{algorithm}

As established by the Frattini argument, this algorithm is applied identically to the normalizer side whenever the library provides the character tables 
for $\N_{\Shat}(P)$ and its corresponding index-2 extension $\N_H(P)$, along with their fusion map.

Hereafter, both in the text and in the implementation, this method of determining the type of each character will be referred to as the \textbf{Library Path}, and shall be used whenever possible.

\section{Computational Framework and Algorithms}
\label{sec:computational_framework}

We provide algorithms to construct the bijection. The approach avoids building groups where possible, falling back to algebraic constructions when library 
tables are absent. 

\subsection{The Block-and-Match Bijection}

After classifying both the $\Shat$ and $\N_{\Shat}(P)$ sides, Algorithm \ref{algo:block_match} constructs the bijection. 

To satisfy Condition (C3), characters are blocked by their central behavior. For Type 1, this is $\lambda_\chi$. For Type 2 orbits, observation 
shows that $\chi^\alpha$ lies over $\lambda_\chi^\alpha$. Thus, the multiset $\{\lambda_\chi, \lambda_{\chi^\alpha}\}$ forms an $\alpha$-orbit 
on $\Irr(\Z(\Shat))$. To satisfy Condition (C4), these blocks are refined into sub-blocks based on $C(\chi)_p$. 

Finally, a matching resolves Condition (C2) by pairing characters positionally after sorting by descending degree. For Type 2 orbits, the orientation of the mapping is resolved strictly so that the characters with identical central characters are mapped to one another.

\begin{algorithm}[H]
\caption{Block-and-Match Bijection}\label{algo:block_match}
\SetKwInOut{Input}{Input}
\SetKwInOut{Output}{Output}
\Input{Classified $p'$-characters of $\Shat$ and $\N_{\Shat}(P)$ with degrees, types, $\lambda$, and $C(\chi)_p$.}
\Output{Explicit bijection ${}^*$ satisfying (C1)--(C4).}

\For{$t \in \{1, 2\}$}{
    \For{each valid central coordinate $c$ (unique $\lambda$ for $t=1$, or unordered pair $\{\lambda_1, \lambda_2\}$ for $t=2$)}{
        \For{each valid conductor $p$-part $v$}{
            Let $\mathcal{B}_{\Shat}$ be the set of characters in $\Shat$ matching type $t$, central coordinate $c$, and conductor $v$\;
            Let $\mathcal{B}_{\N}$ be the set of characters in $\N_{\Shat}(P)$ matching $t$, $c$, and $v$\;
            \If{$|\mathcal{B}_{\Shat}| \neq |\mathcal{B}_{\N}|$}{
                Report non-existence of the bijection\;
            }
            Order the elements of $\mathcal{B}_{\Shat}$ and $\mathcal{B}_{\N}$ in descending sequence by character degree\;
            \For{$k=1$ \KwTo $|\mathcal{B}_{\Shat}|$}{
                Map the orbit $\mathcal{O}_{\Shat}^{(k)} \mapsto \mathcal{O}_{\N}^{(k)}$\;
                \If{$t = 2$}{
                    Orient the mapping $\chi \mapsto \chi^*$ within the orbit such that $\lambda_\chi = \lambda_{\chi^*}$\;
                }
                \If{$\chi^*(1) > \chi(1)$ for any $\chi \in \mathcal{O}_{\Shat}^{(k)}$}{
                    Report a violation of condition (C2)\;
                }
            }
        }
    }
}
\end{algorithm}

\begin{theorem}
\label{thm:biyeccion_correcta}
For a given sporadic group $S$, Problem \ref{prob:mckay} is positive if and only if the algorithm terminates without reporting a violation.
\end{theorem}
\begin{proof}
The map ${}^*$ is a bijection, since $|\mathcal{B}_{\Shat}| = |\mathcal{B}_{\N}|$ is checked for all respective sub-blocks, and it pairs elements 
bijectively within each partition.

\textbf{Condition (C1):} By Corollary \ref{cor:un-alpha} and Section \ref{sec:clifford}, equivariance is equivalent to preserving types and mapping 
$\langle \alpha \rangle$-orbits to $\langle \alpha \rangle$-orbits. The algorithm explicitly partitions characters by type ($t \in \{1, 2\}$) and maps Type 1 
singletons to Type 1 singletons, and Type 2 pairs to Type 2 pairs. Hence, (C1) is satisfied.

\textbf{Condition (C2):} In Step 11, characters within each sub-block are ordered by descending degree. Step 15 explicitly verifies that $\chi^*(1) \le \chi(1)$. 
Since the algorithm terminates without reporting a violation, (C2) is strictly maintained.

\textbf{Condition (C3):} Characters are partitioned into blocks based on their central character behavior $c$. For Type 1 characters, $c = \lambda_\chi$. 
For Type 2 characters, $c = \{\lambda_\chi, \lambda_{\chi^\alpha}\}$. Step 14 ensures that within a Type 2 orbit, the mapping is oriented such that 
$\lambda_\chi = \lambda_{\chi^*}$. Consequently, for all $\chi$, the associated linear character $\lambda$ is preserved, satisfying (C3).

\textbf{Condition (C4):} Characters are further partitioned by the $p$-part of the conductor, $v = C(\chi)_p$. Characters are only paired if they belong 
to the same sub-block, meaning $C(\chi)_p = C(\chi^*)_p$ by definition. This satisfies (C4).
\end{proof}

\subsection{Planning and Resource Hoisting}

A direct approach iterating through primes and computing automorphisms sequentially is inefficient. Algorithm \ref{algo:plan} generates a plan per prime, assigning 
either the \texttt{library} route or the \texttt{standard} group-theoretic route. 

A critical subtlety arises regarding the central characters (Condition C3). To accurately compare the linear characters $\lambda_\chi$ and $\lambda_{\chi^*}$ of $\Z(\Shat)$, the computations for $\Shat$ and $\N_{\Shat}(P)$ must agree on the fixed generator of the cyclic center group $\Z(\Shat)$. The \textbf{Standard Path} constructs a minimal generating set, which may arbitrarily select $z^k$, whereas the \textbf{Library Path} systematically selects the lowest-indexed central class of order $m(S)$. If the Euler totient function satisfies $\varphi(m(S)) > 1$ (i.e., $m(S) \in \{3,4,6,12\}$), this discrepancy risks shifting the indices by an automorphism of $\mathbb{Z}/m(S)$. To prevent this misalignment, Algorithm \ref{algo:plan} enforces a homogeneous plan when $\varphi(m(S)) > 1$, coercing both computations to the \textbf{Standard Path}.

\begin{algorithm}[H]
\caption{Compute Plan and Hoist Resources}\label{algo:plan}
\SetKwInOut{Input}{Input}
\SetKwInOut{Output}{Output}
\Input{A sporadic simple group $S$, its Schur multiplier $m(S)$, and a set of primes $\mathcal{P}$.}
\Output{An execution plan and hoisted global resources.}

\For{each prime $p \in \mathcal{P}$}{
    \eIf{extension tables for $\Shat$ and $\N_{\Shat}(P)$ exist in \normalfont\texttt{CTblLib}}{
        Assign the \texttt{library} route to both $\Shat$ and $\N_{\Shat}(P)$\;
    }{
        Assign the \texttt{standard} route to whichever side lacks tables\;
    }
    \If{the assigned routes are mixed and $\varphi(m(S)) > 1$}{
        Reassign both sides to the \texttt{standard} route to preserve central character generator consistency\;
    }
}
\If{any prime requires the \textnormal{\texttt{standard}} route}{
    Construct the group $\Shat$, $\Aut(\Shat)$, and $\Z(\Shat)$\;
}
\If{any prime utilizes the \textnormal{\texttt{library}} route}{
    Load the necessary character tables for $\Shat$ and its extensions from \texttt{CTblLib}\;
}
Construct the canonical cyclic table $t_Z$ of order $m(S)$ globally\;
\end{algorithm}

\subsection{GAP API Specification}

Although this paper focuses primarily on sporadic groups, the implemented computational framework only requires the simple group to have an outer automorphism group of order at most $2$. To demonstrate this broader applicability, the program includes support for several non-sporadic simple groups meeting this condition (such as alternating groups up to $\mathsf{A}_{19}$, excluding $A_6$ which has $\Out(A_6) = C_2 \times C_2$), which can be evaluated seamlessly using the same API.

The main API function for calling the GAP program follows this structure:
\begin{lstlisting}[language=GAP]
Classify( name_S, prime_list, clifford_verification, show_progress, show_tables )
\end{lstlisting}

The parameters are defined as follows:
\begin{itemize}
    \item \texttt{name\_S}: ATLAS name of the simple group $S$ from the files SporadicData.g or NonSporadicData.g (e.g., \texttt{"HS"}, \texttt{"J2"}, \texttt{"A5"}).
    \item \texttt{prime\_list}: List of primes to classify (e.g., \texttt{[2, 3, 5]}).
    \item \texttt{clifford\_verification}: If \texttt{true} and a prime takes the \textbf{Standard Path}, explicitly build $A$ and verify Clifford witnesses (Algorithm \ref{algo:verification}).
    \item \texttt{show\_progress}: Print step-by-step progress messages.
    \item \texttt{show\_tables}: Display the character tables of $\Shat$, $\Z(\Shat)$, and $\N_{\Shat}(P)$.
\end{itemize}

\subsection{Program Output Structure}

The complete output generated by the program is divided into the following sequential sections:
\begin{enumerate}
    \item \textbf{$\Shat$ and $\Z(\Shat)$ Tables:} Optional display of the character table of $\Shat$ and $\Z(\Shat)$ (if \texttt{show\_tables = true}).
    \item \textbf{Header:} Displays the group name, universal cover group name, Schur multiplier, $|\Out(S)|$, the dispatched routes (\texttt{via\_S}, \texttt{via\_N}, \texttt{verify\_S}, \texttt{verify\_N}), the verification status of each side, the orders of the GAP groups built (if any) and the tables used.
    \item \textbf{Classification for $\Shat$:} Details the classification of $\Irr_{p'}(\Shat)$ and the Clifford verification status.
    \item \textbf{Normalizer Table:} Optional display of the character table of $\N_{\Shat}(P)$ (if \texttt{show\_tables = true}).
    \item \textbf{Classification for the Normalizer:} Details the classification of $\Irr_{p'}(\N_{\Shat}(P))$ and the Clifford verification status.
    \item \textbf{Bijection and Audit:} Presents the bijection and the four-condition audit.
\end{enumerate}

For type-2 pairings whose central characters $\alpha$-orbit has size 2 (i.e., $\alpha$ swaps the lambdas), the output renders the orientation explicitly:

\begin{lstlisting}[language=GAP]
{chi_S[6]@lam=2, chi_S[8]@lam=3} (degree 7) -->
{chi_N[9]@lam=2, chi_N[5]@lam=3} (degree 1)
[case 2, lambdas={2,3}, C_p(S)=4, C_p(N)=4]
\end{lstlisting}

This makes visible which character of the $\Shat$-orbit goes to which character of the $N$-orbit, beyond the unordered pair.

\subsection{Program Overview}

The following is the bird's-eye view of one call to \texttt{Classify}.

\begin{algorithm}[H]
\caption{Program Execution Flow of \texttt{Classify}}\label{algo:flow}
\textbf{1. Plan per prime}: For each $p \in \texttt{prime\_list}$, \texttt{ComputePlanForPrime} decides $\texttt{via\_S} \in \{\text{library}, \text{standard}\}$, $\texttt{via\_N} \in \{\text{library}, \text{standard}\}$ and the verification mode\;
\textbf{2. Hoisted resources}: \texttt{PrecomputeResources} aggregates the plans, then computes \emph{once} only the resources needed by at least one prime: $\Shat$, $\Aut(\Shat)$, $\Inn(\Shat)$, $\Z(\Shat)$, $\mathrm{fus}_Z$, library tables and library $\mathrm{fus}_S$\;
\textbf{3. Per-prime execution}: For each $p$, \texttt{ExecutePrime} dispatches the two sides to the \textbf{Library Path} or \textbf{Standard Path}, optionally verifies via Clifford theory, and produces $\texttt{class\_S}$ and $\texttt{class\_N}$\;
\textbf{4. Bijection and audit}: \texttt{BuildBijection} assembles ${}^{*}$ following the recipe of Theorem \ref{thm:biyeccion_correcta}; \texttt{VerifyMcKayConditions} audits (C1)--(C4)\;
\textbf{5. Output}: \texttt{PrintPrimeResults} prints classification, verification and bijection per prime\;
\end{algorithm}

\subsection{Character Table Library and Nomenclature}

It should be noted that all precomputed character tables utilized in the \texttt{library} method are sourced from the GAP Character Table Library (\texttt{CTblLib}). Within the ATLAS and \texttt{CTblLib} naming conventions, direct products are explicitly denoted with the symbol \texttt{x}. Consequently, a table designated with a \texttt{.2} suffix, such as \texttt{2.M12.2} or \texttt{6.Suz.2}, represents a genuine, non-trivial index-2 extension of $\Shat$, which is the one we need in Theorem \ref{thm:atajo-S}, rather than a trivial direct product $\Shat \times C_2$.

Furthermore, while the \textbf{Library Shortcut} provides significant computational advantages, not all necessary character tables for the universal covers, their index-2 extensions, and their corresponding Sylow normalizers are currently available in \texttt{CTblLib}. This limitation is precisely what necessitates the implementation of the \textbf{Standard Path} in our algorithmic framework to compute the missing cases.

\section{Computer specifications}
\label{sec:PC_specifications}

All results presented were computed using a personal PC with the following specifications:

\begin{table}[h!]
\centering
\begin{tabular}{ll}
\toprule
\textbf{Component} & \textbf{Specification} \\
\midrule
\textbf{Processor} & Intel Core i7-10750H @ 2.60GHz \\
\textbf{Installed RAM} & 16.0 GB \\
\textbf{Graphics Card} & NVIDIA GeForce RTX 2060 with Max-Q Design (6 GB) \\
 & Intel(R) UHD Graphics (128 MB) \\
\textbf{System Type} & 64-bit operating system, x64-based processor \\
\bottomrule
\end{tabular}
\caption{System Specifications}
\label{PC_specifications}
\end{table}

As discussed in Sections \ref{sec:inductive-condition} and \ref{sec:clifford}, taking the \textbf{Standard Path} (Algorithm \ref{algo:standard_path}) or performing the explicit Clifford verification (Algorithm \ref{algo:verification}) carries a significant computational burden. The primary bottleneck in these executions within GAP lies in the computation of the full automorphism group $\Aut(\Shat)$, the set-stabilizer $\Aut(\Shat)_P$, and the inner automorphisms $\Inn(\Shat)_P$. These specific calculations are not only the most time-consuming steps but also the most intensive in terms of memory consumption, which is the primary cause of out-of-memory errors during execution. Additionally, another clear limitation is the need to explicitly construct the groups, which in the largest cases is directly unfeasible.

It is not necessary to examine the largest sporadic groups to encounter these hardware limitations. For example, even for a relatively small group such as the Janko group $J_2$ (whose universal cover $2.J_2$ has order $1,209,600$), the computation using the Standard Path for the Sylow normalizer side takes over two minutes (135 seconds) to complete. For slightly larger but still relatively small groups, such as the Higman--Sims group $\mathrm{HS}$ (whose universal cover $2.\mathrm{HS}$ has order $88,704,000$), this same operation completely exhausts the available memory and causes the execution to crash before completion (as happened when computing the $p=5$ case). In stark contrast, when the \textbf{Library Path} is available, the execution time for groups of similar or much larger size (such as $M_{24}$ or $Co_3$) is reduced to a fraction of a second.

\section*{Acknowledgments}

I would like to express my deepest gratitude to J. Miquel Martinez, not only for agreeing to supervise the collaboration project upon which this entire work is built, but also for his invaluable guidance in writing and publishing these results. I am profoundly grateful for the hours he dedicated to this endeavor. Additionally, I would like to thank G. Navarro for his encouragement in publishing these findings on the arXiv.

\section*{Declaration on the Use of AI}
In accordance with the guidelines of the International Mathematical Union (IMU), the author acknowledges the use of artificial intelligence tools as an assistant for writing and debugging the computational code utilized in this work. The theoretical results, their proofs, and the entirety of the text were written exclusively by the author. The author assumes full responsibility for the content, accuracy, and integrity of this paper.

\newpage
\appendix

\section{Example of an explicit bijection produced}

Below is an example of an explicit classification of irreducible characters and the corresponding bijection for the sporadic group $\mathrm{Fi}_{24}'$ and the prime $p=5$.

\label{app:bijection_example}

\begin{lstlisting}[language=GAP]
########################################################
### T2.16 | Fi24' | primes=[ 5 ] | clifford=true
########################################################

========================================================
=== PRIME p = 5
========================================================
Simple group   = Fi24'
Cover S_hat    = 3.Fi24'
Schur multiplier = 3
|Out(S)| = 2
Route, S_hat side       : library
Route, N_{S_hat}(P) side: library
Verification, S_hat     : implicit
Verification, N         : implicit
Table used for S_hat    : 3.F3+
Table for the S_hat extension : 3.F3+.2
Table for N_{S_hat}(P)  : 3.F3+N5 (library)
Table for N_A(P)        : 3.F3+.2N5

--- CLASSIFICATION OF Irr_{p'}(S_hat) ---
chi[1] degree 1, conductor 1: case 1, lambda=1
chi[2] degree 8671, conductor 1: case 1, lambda=1
chi[3] degree 57477, conductor 1: case 1, lambda=1
chi[4] degree 249458, conductor 1: case 1, lambda=1
chi[5] degree 555611, conductor 1: case 1, lambda=1
chi[8] degree 1666833, conductor 1: case 1, lambda=1
chi[9] degree 4864431, conductor 1: case 1, lambda=1
chi[10] degree 32715683, conductor 1: case 1, lambda=1
chi[13] degree 48893768, conductor 1: case 1, lambda=1
chi[16] degree 79452373, conductor 1: case 1, lambda=1
chi[17] degree 112168056, conductor 1: case 1, lambda=1
chi[19] degree 281380736, conductor 1: case 1, lambda=1
chi[20] degree 415098112, conductor 1: case 1, lambda=1
chi[21] degree 635618984, conductor 1: case 1, lambda=1
chi[23] degree 1112333222, conductor 1: case 1, lambda=1
chi[25] degree 1337276304, conductor 1: case 1, lambda=1
chi[26] degree 1540153692, conductor 1: case 1, lambda=1
chi[27] degree 2346900864, conductor 1: case 1, lambda=1
chi[31] degree 5005499499, conductor 1: case 1, lambda=1
chi[33] degree 6471756928, conductor 1: case 1, lambda=1
chi[35] degree 8529641472, conductor 1: case 1, lambda=1
chi[38] degree 10169903744, conductor 1: case 1, lambda=1
chi[45] degree 17161712568, conductor 1: case 1, lambda=1
chi[46] degree 18481844304, conductor 1: case 2, partner chi[47], lambda=1
chi[47] degree 18481844304, conductor 1: case 2, partner chi[46], lambda=1
chi[51] degree 35594663104, conductor 1: case 1, lambda=1
chi[52] degree 36858678129, conductor 1: case 1, lambda=1
chi[54] degree 38641860608, conductor 1: case 1, lambda=1
chi[55] degree 40043995992, conductor 1: case 1, lambda=1
chi[57] degree 45049495491, conductor 1: case 1, lambda=1
chi[59] degree 54234085491, conductor 1: case 1, lambda=1
chi[63] degree 63831063582, conductor 1: case 1, lambda=1
chi[64] degree 65393917952, conductor 1: case 2, partner chi[65], lambda=1
chi[65] degree 65393917952, conductor 1: case 2, partner chi[64], lambda=1
chi[66] degree 67331776512, conductor 1: case 1, lambda=1
chi[67] degree 71189326208, conductor 1: case 1, lambda=1
chi[68] degree 74887473024, conductor 1: case 1, lambda=1
chi[71] degree 77108871168, conductor 1: case 1, lambda=1
chi[76] degree 118588933386, conductor 1: case 1, lambda=1
chi[77] degree 132390354096, conductor 21: case 2, partner chi[78], lambda=1
chi[78] degree 132390354096, conductor 21: case 2, partner chi[77], lambda=1
chi[79] degree 139317477376, conductor 1: case 1, lambda=1
chi[82] degree 142169187069, conductor 1: case 1, lambda=1
chi[83] degree 142378652416, conductor 1: case 1, lambda=1
chi[84] degree 145650089984, conductor 1: case 1, lambda=1
chi[85] degree 150201655296, conductor 1: case 1, lambda=1
chi[88] degree 156321775827, conductor 1: case 1, lambda=1
chi[89] degree 156321775827, conductor 1: case 1, lambda=1
chi[91] degree 164572397352, conductor 29: case 2, partner chi[92], lambda=1
chi[92] degree 164572397352, conductor 29: case 2, partner chi[91], lambda=1
chi[93] degree 169598100672, conductor 1: case 1, lambda=1
chi[94] degree 178514751987, conductor 1: case 1, lambda=1
chi[95] degree 184117100544, conductor 1: case 1, lambda=1
chi[103] degree 205940550816, conductor 1: case 1, lambda=1
chi[104] degree 222758961152, conductor 1: case 1, lambda=1
chi[107] degree 282049015248, conductor 1: case 1, lambda=1
chi[109] degree 783, conductor 3: case 2, partner chi[110], lambda=2
chi[110] degree 783, conductor 3: case 2, partner chi[109], lambda=3
chi[111] degree 64584, conductor 3: case 2, partner chi[112], lambda=2
chi[112] degree 64584, conductor 3: case 2, partner chi[111], lambda=3
chi[113] degree 306153, conductor 3: case 2, partner chi[114], lambda=2
chi[114] degree 306153, conductor 3: case 2, partner chi[113], lambda=3
chi[115] degree 306153, conductor 3: case 2, partner chi[116], lambda=2
chi[116] degree 306153, conductor 3: case 2, partner chi[115], lambda=3
chi[117] degree 6724809, conductor 3: case 2, partner chi[118], lambda=2
chi[118] degree 6724809, conductor 3: case 2, partner chi[117], lambda=3
chi[121] degree 25356672, conductor 3: case 2, partner chi[122], lambda=2
chi[122] degree 25356672, conductor 3: case 2, partner chi[121], lambda=3
chi[123] degree 43779879, conductor 3: case 2, partner chi[124], lambda=2
chi[124] degree 43779879, conductor 3: case 2, partner chi[123], lambda=3
chi[125] degree 195019461, conductor 3: case 2, partner chi[126], lambda=2
chi[126] degree 195019461, conductor 3: case 2, partner chi[125], lambda=3
chi[127] degree 195019461, conductor 3: case 2, partner chi[128], lambda=2
chi[128] degree 195019461, conductor 3: case 2, partner chi[127], lambda=3
chi[129] degree 203843871, conductor 3: case 2, partner chi[130], lambda=2
chi[130] degree 203843871, conductor 3: case 2, partner chi[129], lambda=3
chi[135] degree 330032934, conductor 3: case 2, partner chi[136], lambda=2
chi[136] degree 330032934, conductor 3: case 2, partner chi[135], lambda=3
chi[137] degree 1050717096, conductor 3: case 2, partner chi[138], lambda=2
chi[138] degree 1050717096, conductor 3: case 2, partner chi[137], lambda=3
chi[147] degree 1349587008, conductor 12: case 2, partner chi[148], lambda=2
chi[148] degree 1349587008, conductor 12: case 2, partner chi[147], lambda=3
chi[149] degree 1349587008, conductor 12: case 2, partner chi[150], lambda=2
chi[150] degree 1349587008, conductor 12: case 2, partner chi[149], lambda=3
chi[151] degree 1553430879, conductor 3: case 2, partner chi[152], lambda=2
chi[152] degree 1553430879, conductor 3: case 2, partner chi[151], lambda=3
chi[157] degree 2801912256, conductor 3: case 2, partner chi[158], lambda=2
chi[158] degree 2801912256, conductor 3: case 2, partner chi[157], lambda=3
chi[161] degree 4290428142, conductor 3: case 2, partner chi[162], lambda=2
chi[162] degree 4290428142, conductor 3: case 2, partner chi[161], lambda=3
chi[165] degree 4620461076, conductor 3: case 2, partner chi[166], lambda=2
chi[166] degree 4620461076, conductor 3: case 2, partner chi[165], lambda=3
chi[175] degree 9456453864, conductor 3: case 2, partner chi[176], lambda=2
chi[176] degree 9456453864, conductor 3: case 2, partner chi[175], lambda=3
chi[183] degree 15016498497, conductor 3: case 2, partner chi[184], lambda=2
chi[184] degree 15016498497, conductor 3: case 2, partner chi[183], lambda=3
chi[187] degree 21096751104, conductor 3: case 2, partner chi[188], lambda=2
chi[188] degree 21096751104, conductor 3: case 2, partner chi[187], lambda=3
chi[189] degree 21122107776, conductor 3: case 2, partner chi[190], lambda=2
chi[190] degree 21122107776, conductor 3: case 2, partner chi[189], lambda=3
chi[191] degree 21842179632, conductor 3: case 2, partner chi[192], lambda=2
chi[192] degree 21842179632, conductor 3: case 2, partner chi[191], lambda=3
chi[197] degree 41785039296, conductor 3: case 2, partner chi[198], lambda=2
chi[198] degree 41785039296, conductor 3: case 2, partner chi[197], lambda=3
chi[199] degree 63831063582, conductor 3: case 2, partner chi[200], lambda=2
chi[200] degree 63831063582, conductor 3: case 2, partner chi[199], lambda=3
chi[203] degree 80256172032, conductor 3: case 2, partner chi[204], lambda=2
chi[204] degree 80256172032, conductor 3: case 2, partner chi[203], lambda=3
chi[205] degree 80256172032, conductor 3: case 2, partner chi[206], lambda=2
chi[206] degree 80256172032, conductor 3: case 2, partner chi[205], lambda=3
chi[213] degree 120131987976, conductor 3: case 2, partner chi[214], lambda=2
chi[214] degree 120131987976, conductor 3: case 2, partner chi[213], lambda=3
chi[215] degree 135605256192, conductor 3: case 2, partner chi[216], lambda=2
chi[216] degree 135605256192, conductor 3: case 2, partner chi[215], lambda=3
chi[217] degree 135605256192, conductor 3: case 2, partner chi[218], lambda=2
chi[218] degree 135605256192, conductor 3: case 2, partner chi[217], lambda=3
chi[219] degree 142169187069, conductor 3: case 2, partner chi[220], lambda=2
chi[220] degree 142169187069, conductor 3: case 2, partner chi[219], lambda=3
chi[227] degree 154455413112, conductor 3: case 2, partner chi[228], lambda=2
chi[228] degree 154455413112, conductor 3: case 2, partner chi[227], lambda=3
chi[231] degree 178514751987, conductor 3: case 2, partner chi[232], lambda=2
chi[232] degree 178514751987, conductor 3: case 2, partner chi[231], lambda=3
chi[237] degree 210555861318, conductor 3: case 2, partner chi[238], lambda=2
chi[238] degree 210555861318, conductor 3: case 2, partner chi[237], lambda=3
chi[239] degree 215861428224, conductor 3: case 2, partner chi[240], lambda=2
chi[240] degree 215861428224, conductor 3: case 2, partner chi[239], lambda=3
chi[245] degree 274587401088, conductor 3: case 2, partner chi[246], lambda=2
chi[246] degree 274587401088, conductor 3: case 2, partner chi[245], lambda=3
chi[247] degree 274910709213, conductor 3: case 2, partner chi[248], lambda=2
chi[248] degree 274910709213, conductor 3: case 2, partner chi[247], lambda=3
chi[249] degree 295709508864, conductor 3: case 2, partner chi[250], lambda=2
chi[250] degree 295709508864, conductor 3: case 2, partner chi[249], lambda=3
chi[251] degree 358644768768, conductor 3: case 2, partner chi[252], lambda=2
chi[252] degree 358644768768, conductor 3: case 2, partner chi[251], lambda=3
chi[255] degree 405445459419, conductor 3: case 2, partner chi[256], lambda=2
chi[256] degree 405445459419, conductor 3: case 2, partner chi[255], lambda=3
Clifford verification S_hat: implicit (classification via library tables).

--- CLASSIFICATION OF Irr_{p'}(N_S_hat(P)) ---
chi[1] degree 1, conductor 1: case 1, lambda=1
chi[2] degree 1, conductor 1: case 1, lambda=1
chi[3] degree 1, conductor 4: case 1, lambda=1
chi[4] degree 1, conductor 4: case 1, lambda=1
chi[5] degree 2, conductor 1: case 1, lambda=1
chi[6] degree 2, conductor 1: case 1, lambda=1
chi[7] degree 2, conductor 1: case 1, lambda=1
chi[8] degree 2, conductor 1: case 1, lambda=1
chi[9] degree 2, conductor 1: case 2, partner chi[10], lambda=1
chi[10] degree 2, conductor 1: case 2, partner chi[9], lambda=1
chi[11] degree 2, conductor 1: case 2, partner chi[12], lambda=1
chi[12] degree 2, conductor 1: case 2, partner chi[11], lambda=1
chi[13] degree 3, conductor 1: case 1, lambda=1
chi[14] degree 3, conductor 1: case 1, lambda=1
chi[15] degree 3, conductor 1: case 1, lambda=1
chi[16] degree 3, conductor 1: case 1, lambda=1
chi[17] degree 3, conductor 12: case 2, partner chi[19], lambda=3
chi[18] degree 3, conductor 12: case 2, partner chi[20], lambda=3
chi[19] degree 3, conductor 12: case 2, partner chi[17], lambda=2
chi[20] degree 3, conductor 12: case 2, partner chi[18], lambda=2
chi[21] degree 3, conductor 3: case 2, partner chi[22], lambda=3
chi[22] degree 3, conductor 3: case 2, partner chi[21], lambda=2
chi[23] degree 3, conductor 3: case 2, partner chi[24], lambda=3
chi[24] degree 3, conductor 3: case 2, partner chi[23], lambda=2
chi[25] degree 3, conductor 4: case 1, lambda=1
chi[26] degree 3, conductor 4: case 1, lambda=1
chi[27] degree 3, conductor 4: case 1, lambda=1
chi[28] degree 3, conductor 4: case 1, lambda=1
chi[29] degree 3, conductor 3: case 2, partner chi[30], lambda=3
chi[30] degree 3, conductor 3: case 2, partner chi[29], lambda=2
chi[31] degree 3, conductor 3: case 2, partner chi[32], lambda=3
chi[32] degree 3, conductor 3: case 2, partner chi[31], lambda=2
chi[33] degree 3, conductor 12: case 2, partner chi[35], lambda=3
chi[34] degree 3, conductor 12: case 2, partner chi[36], lambda=3
chi[35] degree 3, conductor 12: case 2, partner chi[33], lambda=2
chi[36] degree 3, conductor 12: case 2, partner chi[34], lambda=2
chi[37] degree 3, conductor 12: case 2, partner chi[39], lambda=3
chi[38] degree 3, conductor 12: case 2, partner chi[40], lambda=3
chi[39] degree 3, conductor 12: case 2, partner chi[37], lambda=2
chi[40] degree 3, conductor 12: case 2, partner chi[38], lambda=2
chi[41] degree 3, conductor 3: case 2, partner chi[42], lambda=3
chi[42] degree 3, conductor 3: case 2, partner chi[41], lambda=2
chi[43] degree 3, conductor 3: case 2, partner chi[44], lambda=3
chi[44] degree 3, conductor 3: case 2, partner chi[43], lambda=2
chi[45] degree 6, conductor 1: case 1, lambda=1
chi[46] degree 6, conductor 1: case 1, lambda=1
chi[47] degree 6, conductor 1: case 1, lambda=1
chi[48] degree 6, conductor 1: case 1, lambda=1
chi[49] degree 6, conductor 3: case 2, partner chi[50], lambda=3
chi[50] degree 6, conductor 3: case 2, partner chi[49], lambda=2
chi[51] degree 6, conductor 3: case 2, partner chi[52], lambda=3
chi[52] degree 6, conductor 3: case 2, partner chi[51], lambda=2
chi[53] degree 6, conductor 3: case 2, partner chi[54], lambda=3
chi[54] degree 6, conductor 3: case 2, partner chi[53], lambda=2
chi[55] degree 6, conductor 3: case 2, partner chi[56], lambda=3
chi[56] degree 6, conductor 3: case 2, partner chi[55], lambda=2
chi[57] degree 9, conductor 1: case 1, lambda=1
chi[58] degree 9, conductor 1: case 1, lambda=1
chi[59] degree 9, conductor 4: case 1, lambda=1
chi[60] degree 9, conductor 4: case 1, lambda=1
chi[61] degree 9, conductor 12: case 2, partner chi[63], lambda=3
chi[62] degree 9, conductor 12: case 2, partner chi[64], lambda=3
chi[63] degree 9, conductor 12: case 2, partner chi[61], lambda=2
chi[64] degree 9, conductor 12: case 2, partner chi[62], lambda=2
chi[65] degree 9, conductor 3: case 2, partner chi[66], lambda=3
chi[66] degree 9, conductor 3: case 2, partner chi[65], lambda=2
chi[67] degree 9, conductor 3: case 2, partner chi[68], lambda=3
chi[68] degree 9, conductor 3: case 2, partner chi[67], lambda=2
chi[69] degree 2, conductor 4: case 1, lambda=1
chi[70] degree 2, conductor 4: case 1, lambda=1
chi[71] degree 2, conductor 4: case 1, lambda=1
chi[72] degree 2, conductor 4: case 1, lambda=1
chi[73] degree 4, conductor 4: case 1, lambda=1
chi[74] degree 4, conductor 4: case 1, lambda=1
chi[75] degree 4, conductor 4: case 1, lambda=1
chi[76] degree 4, conductor 4: case 1, lambda=1
chi[77] degree 4, conductor 4: case 2, partner chi[79], lambda=1
chi[78] degree 4, conductor 4: case 2, partner chi[80], lambda=1
chi[79] degree 4, conductor 4: case 2, partner chi[77], lambda=1
chi[80] degree 4, conductor 4: case 2, partner chi[78], lambda=1
chi[81] degree 6, conductor 12: case 2, partner chi[83], lambda=3
chi[82] degree 6, conductor 12: case 2, partner chi[84], lambda=3
chi[83] degree 6, conductor 12: case 2, partner chi[81], lambda=2
chi[84] degree 6, conductor 12: case 2, partner chi[82], lambda=2
chi[85] degree 6, conductor 12: case 2, partner chi[87], lambda=3
chi[86] degree 6, conductor 12: case 2, partner chi[88], lambda=3
chi[87] degree 6, conductor 12: case 2, partner chi[85], lambda=2
chi[88] degree 6, conductor 12: case 2, partner chi[86], lambda=2
chi[89] degree 6, conductor 4: case 1, lambda=1
chi[90] degree 6, conductor 4: case 1, lambda=1
chi[91] degree 6, conductor 4: case 1, lambda=1
chi[92] degree 6, conductor 4: case 1, lambda=1
chi[93] degree 6, conductor 12: case 2, partner chi[95], lambda=3
chi[94] degree 6, conductor 12: case 2, partner chi[96], lambda=3
chi[95] degree 6, conductor 12: case 2, partner chi[93], lambda=2
chi[96] degree 6, conductor 12: case 2, partner chi[94], lambda=2
chi[97] degree 6, conductor 12: case 2, partner chi[99], lambda=3
chi[98] degree 6, conductor 12: case 2, partner chi[100], lambda=3
chi[99] degree 6, conductor 12: case 2, partner chi[97], lambda=2
chi[100] degree 6, conductor 12: case 2, partner chi[98], lambda=2
chi[101] degree 12, conductor 4: case 1, lambda=1
chi[102] degree 12, conductor 4: case 1, lambda=1
chi[103] degree 12, conductor 12: case 2, partner chi[105], lambda=3
chi[104] degree 12, conductor 12: case 2, partner chi[106], lambda=3
chi[105] degree 12, conductor 12: case 2, partner chi[103], lambda=2
chi[106] degree 12, conductor 12: case 2, partner chi[104], lambda=2
chi[107] degree 24, conductor 1: case 1, lambda=1
chi[108] degree 24, conductor 1: case 1, lambda=1
chi[109] degree 24, conductor 4: case 1, lambda=1
chi[110] degree 24, conductor 4: case 1, lambda=1
chi[111] degree 24, conductor 12: case 2, partner chi[113], lambda=3
chi[112] degree 24, conductor 12: case 2, partner chi[114], lambda=3
chi[113] degree 24, conductor 12: case 2, partner chi[111], lambda=2
chi[114] degree 24, conductor 12: case 2, partner chi[112], lambda=2
chi[115] degree 24, conductor 3: case 2, partner chi[116], lambda=3
chi[116] degree 24, conductor 3: case 2, partner chi[115], lambda=2
chi[117] degree 24, conductor 3: case 2, partner chi[118], lambda=3
chi[118] degree 24, conductor 3: case 2, partner chi[117], lambda=2
chi[119] degree 48, conductor 1: case 1, lambda=1
chi[120] degree 48, conductor 1: case 1, lambda=1
chi[121] degree 48, conductor 3: case 2, partner chi[122], lambda=3
chi[122] degree 48, conductor 3: case 2, partner chi[121], lambda=2
chi[123] degree 48, conductor 3: case 2, partner chi[124], lambda=3
chi[124] degree 48, conductor 3: case 2, partner chi[123], lambda=2
chi[125] degree 72, conductor 1: case 1, lambda=1
chi[126] degree 72, conductor 1: case 1, lambda=1
chi[127] degree 72, conductor 4: case 1, lambda=1
chi[128] degree 72, conductor 4: case 1, lambda=1
chi[129] degree 72, conductor 3: case 2, partner chi[130], lambda=3
chi[130] degree 72, conductor 3: case 2, partner chi[129], lambda=2
chi[131] degree 72, conductor 3: case 2, partner chi[132], lambda=3
chi[132] degree 72, conductor 3: case 2, partner chi[131], lambda=2
chi[133] degree 72, conductor 12: case 2, partner chi[135], lambda=3
chi[134] degree 72, conductor 12: case 2, partner chi[136], lambda=3
chi[135] degree 72, conductor 12: case 2, partner chi[133], lambda=2
chi[136] degree 72, conductor 12: case 2, partner chi[134], lambda=2
Clifford verification N_S_hat(P): implicit (classification via library tables).

--- BIJECTION Irr_{p'}(S_hat) <-> Irr_{p'}(N_S_hat(P)) ---
Case-1 orbits: 48 | Case-2 orbits: 44

  chi_S[107] (degree 282049015248) --> chi_N[125] (degree 72)  [case 1, lambda=1, C_p(S)=1, C_p(N)=1]
  chi_S[104] (degree 222758961152) --> chi_N[126] (degree 72)  [case 1, lambda=1, C_p(S)=1, C_p(N)=1]
  chi_S[103] (degree 205940550816) --> chi_N[127] (degree 72)  [case 1, lambda=1, C_p(S)=1, C_p(N)=1]
  chi_S[95] (degree 184117100544) --> chi_N[128] (degree 72)  [case 1, lambda=1, C_p(S)=1, C_p(N)=1]
  chi_S[94] (degree 178514751987) --> chi_N[119] (degree 48)  [case 1, lambda=1, C_p(S)=1, C_p(N)=1]
  chi_S[93] (degree 169598100672) --> chi_N[120] (degree 48)  [case 1, lambda=1, C_p(S)=1, C_p(N)=1]
  chi_S[88] (degree 156321775827) --> chi_N[107] (degree 24)  [case 1, lambda=1, C_p(S)=1, C_p(N)=1]
  chi_S[89] (degree 156321775827) --> chi_N[108] (degree 24)  [case 1, lambda=1, C_p(S)=1, C_p(N)=1]
  chi_S[85] (degree 150201655296) --> chi_N[109] (degree 24)  [case 1, lambda=1, C_p(S)=1, C_p(N)=1]
  chi_S[84] (degree 145650089984) --> chi_N[110] (degree 24)  [case 1, lambda=1, C_p(S)=1, C_p(N)=1]
  chi_S[83] (degree 142378652416) --> chi_N[101] (degree 12)  [case 1, lambda=1, C_p(S)=1, C_p(N)=1]
  chi_S[82] (degree 142169187069) --> chi_N[102] (degree 12)  [case 1, lambda=1, C_p(S)=1, C_p(N)=1]
  chi_S[79] (degree 139317477376) --> chi_N[57] (degree 9)  [case 1, lambda=1, C_p(S)=1, C_p(N)=1]
  chi_S[76] (degree 118588933386) --> chi_N[58] (degree 9)  [case 1, lambda=1, C_p(S)=1, C_p(N)=1]
  chi_S[71] (degree 77108871168) --> chi_N[59] (degree 9)  [case 1, lambda=1, C_p(S)=1, C_p(N)=1]
  chi_S[68] (degree 74887473024) --> chi_N[60] (degree 9)  [case 1, lambda=1, C_p(S)=1, C_p(N)=1]
  chi_S[67] (degree 71189326208) --> chi_N[45] (degree 6)  [case 1, lambda=1, C_p(S)=1, C_p(N)=1]
  chi_S[66] (degree 67331776512) --> chi_N[46] (degree 6)  [case 1, lambda=1, C_p(S)=1, C_p(N)=1]
  chi_S[63] (degree 63831063582) --> chi_N[47] (degree 6)  [case 1, lambda=1, C_p(S)=1, C_p(N)=1]
  chi_S[59] (degree 54234085491) --> chi_N[48] (degree 6)  [case 1, lambda=1, C_p(S)=1, C_p(N)=1]
  chi_S[57] (degree 45049495491) --> chi_N[89] (degree 6)  [case 1, lambda=1, C_p(S)=1, C_p(N)=1]
  chi_S[55] (degree 40043995992) --> chi_N[90] (degree 6)  [case 1, lambda=1, C_p(S)=1, C_p(N)=1]
  chi_S[54] (degree 38641860608) --> chi_N[91] (degree 6)  [case 1, lambda=1, C_p(S)=1, C_p(N)=1]
  chi_S[52] (degree 36858678129) --> chi_N[92] (degree 6)  [case 1, lambda=1, C_p(S)=1, C_p(N)=1]
  chi_S[51] (degree 35594663104) --> chi_N[73] (degree 4)  [case 1, lambda=1, C_p(S)=1, C_p(N)=1]
  chi_S[45] (degree 17161712568) --> chi_N[74] (degree 4)  [case 1, lambda=1, C_p(S)=1, C_p(N)=1]
  chi_S[38] (degree 10169903744) --> chi_N[75] (degree 4)  [case 1, lambda=1, C_p(S)=1, C_p(N)=1]
  chi_S[35] (degree 8529641472) --> chi_N[76] (degree 4)  [case 1, lambda=1, C_p(S)=1, C_p(N)=1]
  chi_S[33] (degree 6471756928) --> chi_N[13] (degree 3)  [case 1, lambda=1, C_p(S)=1, C_p(N)=1]
  chi_S[31] (degree 5005499499) --> chi_N[14] (degree 3)  [case 1, lambda=1, C_p(S)=1, C_p(N)=1]
  chi_S[27] (degree 2346900864) --> chi_N[15] (degree 3)  [case 1, lambda=1, C_p(S)=1, C_p(N)=1]
  chi_S[26] (degree 1540153692) --> chi_N[16] (degree 3)  [case 1, lambda=1, C_p(S)=1, C_p(N)=1]
  chi_S[25] (degree 1337276304) --> chi_N[25] (degree 3)  [case 1, lambda=1, C_p(S)=1, C_p(N)=1]
  chi_S[23] (degree 1112333222) --> chi_N[26] (degree 3)  [case 1, lambda=1, C_p(S)=1, C_p(N)=1]
  chi_S[21] (degree 635618984) --> chi_N[27] (degree 3)  [case 1, lambda=1, C_p(S)=1, C_p(N)=1]
  chi_S[20] (degree 415098112) --> chi_N[28] (degree 3)  [case 1, lambda=1, C_p(S)=1, C_p(N)=1]
  chi_S[19] (degree 281380736) --> chi_N[5] (degree 2)  [case 1, lambda=1, C_p(S)=1, C_p(N)=1]
  chi_S[17] (degree 112168056) --> chi_N[6] (degree 2)  [case 1, lambda=1, C_p(S)=1, C_p(N)=1]
  chi_S[16] (degree 79452373) --> chi_N[7] (degree 2)  [case 1, lambda=1, C_p(S)=1, C_p(N)=1]
  chi_S[13] (degree 48893768) --> chi_N[8] (degree 2)  [case 1, lambda=1, C_p(S)=1, C_p(N)=1]
  chi_S[10] (degree 32715683) --> chi_N[69] (degree 2)  [case 1, lambda=1, C_p(S)=1, C_p(N)=1]
  chi_S[9] (degree 4864431) --> chi_N[70] (degree 2)  [case 1, lambda=1, C_p(S)=1, C_p(N)=1]
  chi_S[8] (degree 1666833) --> chi_N[71] (degree 2)  [case 1, lambda=1, C_p(S)=1, C_p(N)=1]
  chi_S[5] (degree 555611) --> chi_N[72] (degree 2)  [case 1, lambda=1, C_p(S)=1, C_p(N)=1]
  chi_S[4] (degree 249458) --> chi_N[1] (degree 1)  [case 1, lambda=1, C_p(S)=1, C_p(N)=1]
  chi_S[3] (degree 57477) --> chi_N[2] (degree 1)  [case 1, lambda=1, C_p(S)=1, C_p(N)=1]
  chi_S[2] (degree 8671) --> chi_N[3] (degree 1)  [case 1, lambda=1, C_p(S)=1, C_p(N)=1]
  chi_S[1] (degree 1) --> chi_N[4] (degree 1)  [case 1, lambda=1, C_p(S)=1, C_p(N)=1]
  {chi_S[91], chi_S[92]} (degree 164572397352) --> {chi_N[77], chi_N[79]} (degree 4)  [case 2, lambda=1, C_p(S)=1, C_p(N)=1]
  {chi_S[77], chi_S[78]} (degree 132390354096) --> {chi_N[78], chi_N[80]} (degree 4)  [case 2, lambda=1, C_p(S)=1, C_p(N)=1]
  {chi_S[64], chi_S[65]} (degree 65393917952) --> {chi_N[9], chi_N[10]} (degree 2)  [case 2, lambda=1, C_p(S)=1, C_p(N)=1]
  {chi_S[46], chi_S[47]} (degree 18481844304) --> {chi_N[11], chi_N[12]} (degree 2)  [case 2, lambda=1, C_p(S)=1, C_p(N)=1]
  {chi_S[255]@lam=2, chi_S[256]@lam=3} (degree 405445459419) --> {chi_N[130]@lam=2, chi_N[129]@lam=3} (degree 72)  [case 2, lambdas={2,3}, C_p(S)=1, C_p(N)=1]
  {chi_S[251]@lam=2, chi_S[252]@lam=3} (degree 358644768768) --> {chi_N[132]@lam=2, chi_N[131]@lam=3} (degree 72)  [case 2, lambdas={2,3}, C_p(S)=1, C_p(N)=1]
  {chi_S[249]@lam=2, chi_S[250]@lam=3} (degree 295709508864) --> {chi_N[135]@lam=2, chi_N[133]@lam=3} (degree 72)  [case 2, lambdas={2,3}, C_p(S)=1, C_p(N)=1]
  {chi_S[247]@lam=2, chi_S[248]@lam=3} (degree 274910709213) --> {chi_N[136]@lam=2, chi_N[134]@lam=3} (degree 72)  [case 2, lambdas={2,3}, C_p(S)=1, C_p(N)=1]
  {chi_S[245]@lam=2, chi_S[246]@lam=3} (degree 274587401088) --> {chi_N[122]@lam=2, chi_N[121]@lam=3} (degree 48)  [case 2, lambdas={2,3}, C_p(S)=1, C_p(N)=1]
  {chi_S[239]@lam=2, chi_S[240]@lam=3} (degree 215861428224) --> {chi_N[124]@lam=2, chi_N[123]@lam=3} (degree 48)  [case 2, lambdas={2,3}, C_p(S)=1, C_p(N)=1]
  {chi_S[237]@lam=2, chi_S[238]@lam=3} (degree 210555861318) --> {chi_N[113]@lam=2, chi_N[111]@lam=3} (degree 24)  [case 2, lambdas={2,3}, C_p(S)=1, C_p(N)=1]
  {chi_S[231]@lam=2, chi_S[232]@lam=3} (degree 178514751987) --> {chi_N[114]@lam=2, chi_N[112]@lam=3} (degree 24)  [case 2, lambdas={2,3}, C_p(S)=1, C_p(N)=1]
  {chi_S[227]@lam=2, chi_S[228]@lam=3} (degree 154455413112) --> {chi_N[116]@lam=2, chi_N[115]@lam=3} (degree 24)  [case 2, lambdas={2,3}, C_p(S)=1, C_p(N)=1]
  {chi_S[219]@lam=2, chi_S[220]@lam=3} (degree 142169187069) --> {chi_N[118]@lam=2, chi_N[117]@lam=3} (degree 24)  [case 2, lambdas={2,3}, C_p(S)=1, C_p(N)=1]
  {chi_S[215]@lam=2, chi_S[216]@lam=3} (degree 135605256192) --> {chi_N[105]@lam=2, chi_N[103]@lam=3} (degree 12)  [case 2, lambdas={2,3}, C_p(S)=1, C_p(N)=1]
  {chi_S[217]@lam=2, chi_S[218]@lam=3} (degree 135605256192) --> {chi_N[106]@lam=2, chi_N[104]@lam=3} (degree 12)  [case 2, lambdas={2,3}, C_p(S)=1, C_p(N)=1]
  {chi_S[213]@lam=2, chi_S[214]@lam=3} (degree 120131987976) --> {chi_N[63]@lam=2, chi_N[61]@lam=3} (degree 9)  [case 2, lambdas={2,3}, C_p(S)=1, C_p(N)=1]
  {chi_S[203]@lam=2, chi_S[204]@lam=3} (degree 80256172032) --> {chi_N[64]@lam=2, chi_N[62]@lam=3} (degree 9)  [case 2, lambdas={2,3}, C_p(S)=1, C_p(N)=1]
  {chi_S[205]@lam=2, chi_S[206]@lam=3} (degree 80256172032) --> {chi_N[66]@lam=2, chi_N[65]@lam=3} (degree 9)  [case 2, lambdas={2,3}, C_p(S)=1, C_p(N)=1]
  {chi_S[199]@lam=2, chi_S[200]@lam=3} (degree 63831063582) --> {chi_N[68]@lam=2, chi_N[67]@lam=3} (degree 9)  [case 2, lambdas={2,3}, C_p(S)=1, C_p(N)=1]
  {chi_S[197]@lam=2, chi_S[198]@lam=3} (degree 41785039296) --> {chi_N[50]@lam=2, chi_N[49]@lam=3} (degree 6)  [case 2, lambdas={2,3}, C_p(S)=1, C_p(N)=1]
  {chi_S[191]@lam=2, chi_S[192]@lam=3} (degree 21842179632) --> {chi_N[52]@lam=2, chi_N[51]@lam=3} (degree 6)  [case 2, lambdas={2,3}, C_p(S)=1, C_p(N)=1]
  {chi_S[189]@lam=2, chi_S[190]@lam=3} (degree 21122107776) --> {chi_N[54]@lam=2, chi_N[53]@lam=3} (degree 6)  [case 2, lambdas={2,3}, C_p(S)=1, C_p(N)=1]
  {chi_S[187]@lam=2, chi_S[188]@lam=3} (degree 21096751104) --> {chi_N[56]@lam=2, chi_N[55]@lam=3} (degree 6)  [case 2, lambdas={2,3}, C_p(S)=1, C_p(N)=1]
  {chi_S[183]@lam=2, chi_S[184]@lam=3} (degree 15016498497) --> {chi_N[83]@lam=2, chi_N[81]@lam=3} (degree 6)  [case 2, lambdas={2,3}, C_p(S)=1, C_p(N)=1]
  {chi_S[175]@lam=2, chi_S[176]@lam=3} (degree 9456453864) --> {chi_N[84]@lam=2, chi_N[82]@lam=3} (degree 6)  [case 2, lambdas={2,3}, C_p(S)=1, C_p(N)=1]
  {chi_S[165]@lam=2, chi_S[166]@lam=3} (degree 4620461076) --> {chi_N[87]@lam=2, chi_N[85]@lam=3} (degree 6)  [case 2, lambdas={2,3}, C_p(S)=1, C_p(N)=1]
  {chi_S[161]@lam=2, chi_S[162]@lam=3} (degree 4290428142) --> {chi_N[88]@lam=2, chi_N[86]@lam=3} (degree 6)  [case 2, lambdas={2,3}, C_p(S)=1, C_p(N)=1]
  {chi_S[157]@lam=2, chi_S[158]@lam=3} (degree 2801912256) --> {chi_N[95]@lam=2, chi_N[93]@lam=3} (degree 6)  [case 2, lambdas={2,3}, C_p(S)=1, C_p(N)=1]
  {chi_S[151]@lam=2, chi_S[152]@lam=3} (degree 1553430879) --> {chi_N[96]@lam=2, chi_N[94]@lam=3} (degree 6)  [case 2, lambdas={2,3}, C_p(S)=1, C_p(N)=1]
  {chi_S[147]@lam=2, chi_S[148]@lam=3} (degree 1349587008) --> {chi_N[99]@lam=2, chi_N[97]@lam=3} (degree 6)  [case 2, lambdas={2,3}, C_p(S)=1, C_p(N)=1]
  {chi_S[149]@lam=2, chi_S[150]@lam=3} (degree 1349587008) --> {chi_N[100]@lam=2, chi_N[98]@lam=3} (degree 6)  [case 2, lambdas={2,3}, C_p(S)=1, C_p(N)=1]
  {chi_S[137]@lam=2, chi_S[138]@lam=3} (degree 1050717096) --> {chi_N[19]@lam=2, chi_N[17]@lam=3} (degree 3)  [case 2, lambdas={2,3}, C_p(S)=1, C_p(N)=1]
  {chi_S[135]@lam=2, chi_S[136]@lam=3} (degree 330032934) --> {chi_N[20]@lam=2, chi_N[18]@lam=3} (degree 3)  [case 2, lambdas={2,3}, C_p(S)=1, C_p(N)=1]
  {chi_S[129]@lam=2, chi_S[130]@lam=3} (degree 203843871) --> {chi_N[22]@lam=2, chi_N[21]@lam=3} (degree 3)  [case 2, lambdas={2,3}, C_p(S)=1, C_p(N)=1]
  {chi_S[125]@lam=2, chi_S[126]@lam=3} (degree 195019461) --> {chi_N[24]@lam=2, chi_N[23]@lam=3} (degree 3)  [case 2, lambdas={2,3}, C_p(S)=1, C_p(N)=1]
  {chi_S[127]@lam=2, chi_S[128]@lam=3} (degree 195019461) --> {chi_N[30]@lam=2, chi_N[29]@lam=3} (degree 3)  [case 2, lambdas={2,3}, C_p(S)=1, C_p(N)=1]
  {chi_S[123]@lam=2, chi_S[124]@lam=3} (degree 43779879) --> {chi_N[32]@lam=2, chi_N[31]@lam=3} (degree 3)  [case 2, lambdas={2,3}, C_p(S)=1, C_p(N)=1]
  {chi_S[121]@lam=2, chi_S[122]@lam=3} (degree 25356672) --> {chi_N[35]@lam=2, chi_N[33]@lam=3} (degree 3)  [case 2, lambdas={2,3}, C_p(S)=1, C_p(N)=1]
  {chi_S[117]@lam=2, chi_S[118]@lam=3} (degree 6724809) --> {chi_N[36]@lam=2, chi_N[34]@lam=3} (degree 3)  [case 2, lambdas={2,3}, C_p(S)=1, C_p(N)=1]
  {chi_S[113]@lam=2, chi_S[114]@lam=3} (degree 306153) --> {chi_N[39]@lam=2, chi_N[37]@lam=3} (degree 3)  [case 2, lambdas={2,3}, C_p(S)=1, C_p(N)=1]
  {chi_S[115]@lam=2, chi_S[116]@lam=3} (degree 306153) --> {chi_N[40]@lam=2, chi_N[38]@lam=3} (degree 3)  [case 2, lambdas={2,3}, C_p(S)=1, C_p(N)=1]
  {chi_S[111]@lam=2, chi_S[112]@lam=3} (degree 64584) --> {chi_N[42]@lam=2, chi_N[41]@lam=3} (degree 3)  [case 2, lambdas={2,3}, C_p(S)=1, C_p(N)=1]
  {chi_S[109]@lam=2, chi_S[110]@lam=3} (degree 783) --> {chi_N[44]@lam=2, chi_N[43]@lam=3} (degree 3)  [case 2, lambdas={2,3}, C_p(S)=1, C_p(N)=1]
\end{lstlisting}

\end{document}